\documentclass{amsart}

\usepackage{amsthm,amssymb,latexsym,amsmath}
\usepackage{mathrsfs}
\usepackage{graphicx}
\input xy
\xyoption{all}
\usepackage[all]{xy}

\newcommand{\CC}{\mbox{${\mathbb C}$}}

\newcommand{\fol}{\mbox{${\mathscr F}$}}

\newcommand{\D}{\mbox{${\mathscr D}$}}

\newcommand{\NNN}{{\mathcal{N}}}

\newcommand{\OO}{{\mathcal{O}}}

\newcommand{\PP}{\mathbb{P}}

\newtheorem{lema}{Lemma}[section]
\newtheorem{cor}[lema]{Corollary}
\newtheorem{obs}[lema]{Remark}
\newtheorem{teo}[lema]{Theorem}

\newtheorem{prop}[lema]{Proposition}

\theoremstyle{definition}

\newtheorem{defi}[lema]{Definition}
\newtheorem{exe}[lema]{Example}

\begin{document}

\title{GSV-index for  holomorphic Pfaff Systems}
\hyphenation{ho-mo-lo-gi-cal}
\hyphenation{fo-lia-tion}


\begin{abstract}
In this work we introduce  a  GSV  type index for varieties  invariant by   holomorphic Pfaff systems (possibly non locally
decomposables)  on  projective  manifolds. 
 We prove a  non-negativity property   for    the index. As an application, we prove that the non-negativity of the GSV-index gives us an obstruction to the solution of the Poincar\'e problem for Pfaff systems on projectives spaces.
\end{abstract}

\author{Maur\'icio Corr\^ea}
\address{Maur\'icio Corr\^ea \\ ICEx - UFMG \\
Departamento de Matem\'atica \\
Av. Ant\^onio Carlos 6627 \\
30123-970 Belo Horizonte MG, Brazil } \email{mauriciojr@ufmg.br}
\author{Diogo da Silva Machado}
\address{ Diogo da Silva Machado \\ DMA - UFV , Av P.H. Rolfs, s/n, Campus Universit\'ario, Vicosa MG, Brazil, CEP 36570-900}
\email{diogo.machado@ufv.com.br}
\thanks{ }

\maketitle

\section{Introduction}

The GSV-index for vector fields tangent to hypersurfaces with isolated singularities  was introduced by X. G\'omez-Mont, J. Seade and A. Verjovsky \cite{GSV}  generalizing   the Poincar\'e-Hopf index. 
The concept of GSV-index for vector fields tangent to complete intersections was  extended   by J. Seade and T. Suwa in  \cite{SeaSuw1, SeaSuw2} and   J. -P. Brasselet, J. Seade and T. Suwa in  \cite{a03}.

 D. Lehmann, M. G. Soares and T. Suwa  introduced in \cite{a07} the {\it virtual index}  for vector fields on complex analytic varieties  via Chern-Weil theory.
They showed that  this index  coincides with the GSV-index if the variety is a local complete intersection variety with isolated singularities. 
  X. Gomez-Mont in  \cite{a08} defined the {\it homological index}  for holomorphic vector fields on arbitrary varieties with an isolated
normal singularity, and it coincides with the GSV-index when the variety in
question is complete intersection. Recently, 
T. Suwa in \cite{Suw3}  gave a new interpretation of the  GSV-index as  a  residue arising from a certain localization of the Chern class of the ambient tangent
bundle.

M. Brunella in \cite{Bru}  introduced the GSV-index for 1-dimensional singular foliations in complex surfaces in terms of the germs of  $1$-forms inducing the foliation and   established a relation  between  the GSV-index  with the 
{\it Khanedani-Suwa variational  index}  \cite{KS} and {\it Camacho-Sad index}  \cite{CS}. Moreover,   he showed that the non-negativity of the GSV-index  is the obstruction  to the solution of the Poincar\'e problem in complex compact surfaces. We recall that  the   {\it Poincar\'e problem}   is a question proposed by H. Poincar\'e in \cite{Poin}    of bounding
the degree of algebraic solutions of an algebraic differential equation on the complex
plane. Many authors have been working on the  Poincar\'e problem  and   its generalizations for  Pfaff systems, 
see for instance the papers by D.  Cerveau and A.  Lins Neto  \cite{Cerv1},     M. G. Soares \cite{So1},  M. Brunella
and L. G. Mendes \cite{BruMe}, E. Esteves and S. L. Kleiman \cite{EK}, V. Cavalier and D. Lehmann \cite{CL}, E. Esteves and J. D. A.  Cruz \cite{EstCruz},    M. Corr\^ea and M. Jardim \cite{CJ}, and M. Corr\^ea and M. G. Soares \cite{CS1,CS2}.

In this work, our main goal is to  introduce  a  GSV  type index for Pfaff systems on projective  manifolds and demonstrate some of its important  properties. More precisely, we prove the following Theorem.

\begin{teo}\label{prop9}
Let $X$ be a projective manifold and  $V\subset X$ a reduced local complete intersection subvariety  of codimension $k$ invariant by  a Pfaff system, of rank $k$, induced by a twisted form $\omega\in \mathrm H^0(X,\Omega_X^k\otimes \NNN)$.  Then the following hold:
\begin{itemize}
\item[(a)] 
there exists a  complex number $\mathrm{GSV}(\omega, V, S_i)$ which  depends only on the local   representatives of  $\omega, V$ and  $S_i$;
\\ 
\item[(b)]  
if     $\mathrm{Sing}(\omega,V): =\mathrm{Sing}(\omega)\cap V$ has codimension one in $V$, then  the following formula holds 
$$
\sum_i \mathrm{GSV}(\omega, V, S_i) [S_i] = c_1([\NNN \otimes \det (N_{V/X})^{-1}])|_V\frown [V],  
$$
where $S_i$ denotes an irreducible component of $\mathrm{Sing}(\omega,V)$ and  $N_{V/X}$ is the normal sheaf of the subvariety  $V$.
\end{itemize}
\end{teo}

The next result says us how   to calculate the GSV-index which  can be compared with  Suwa's  formula in \cite[Proposition 5.1 ]{Suw3}.

\begin{teo}\label{teo404} Setting as in Theorem \ref{prop9}.
Let  $x\in V$, let 
$\{ f_{ 1}=\cdots= f_{ k} = 0\}$
be a local equation of $V$ in a neighborhood $U $ of $x$ and  consider  
$$
\omega_{|U} = \sum_{\mid I\mid = k}a_I dZ_I, \,\,\,\,\,\mbox{$a_I\in \OO(U )$}.
$$
the  holomorphic $k$-form inducing the  Pfaff system  $\omega\in \mathrm H^0(X,\Omega_X^k\otimes \NNN)$ on $U $. 
Then the following formula holds 
\begin{eqnarray} \label{formula-gsv}
\mathrm{GSV}(\omega,V,S_i) = \mathrm{ord}_{S_i}(a_{I\,}|_V) - \mathrm{ord}_{S_i}(\Delta_{I\,}|_V).
\end{eqnarray}
\noindent where $\Delta_I$ is the $k\times k$ minor of the Jacobian matrix
  $ \mathrm{Jac}(f_{ 1},\ldots,f_{ k})$, corresponding to the multi-index $I$.
\end{teo}

The formula (\ref{formula-gsv})  allows us  to prove the following   non-negativity properties    for    the index:
\begin{itemize}
\item[(i)]  If $S_i\cap \mathrm{Sing}(V)= \emptyset$, then $\mathrm{GSV}(\omega,V,S_i)\geq 0$. 

\item[(ii)] If $V$ is  smooth, then $\mathrm{GSV}(\omega,V,S_i)>0$. 
\end{itemize}
See corollaries \ref{4p4} and  \ref{4p10} in the section \ref{section-coro}.

As an application, we prove,  in the section \ref{sec23},  that the non-negativity of the GSV-index gives an  obstruction to the solution of the Poincar\'e Problem for Pfaff systems on projective complex space. More precisely, 
let $\omega \in \mathrm H^0(\PP^n,\Omega_{\PP^n}^k(d+k+1))$ be a holomorphic  Pfaff system of rank $k$ and degree $ d$. Let $V\subset \PP^n$  be a reduced   complete intersection variety, of codimension $k$ and multidegree $(d_1,\dots,d_{k})$ ,  invariant  by $\omega$. Suppose that  $\mathrm{Sing}(\omega,V)$ has codimension one in $V$,  then
\begin{eqnarray} \nonumber 
\sum_i \mathrm{GSV}(\omega,V,S_i)\deg(S_i) \,\,\,=\,\,\, [d+k+1 - (d_1+\cdots+d_k)]\,\cdot (d_1\cdots d_k),
\end{eqnarray}
where $S_i$ denotes an irreducible component of $\mathrm{Sing}(\omega,V)$.
Therefore,   if $\mathrm{GSV}(\omega,V,S_i)\geq 0$, for all $i$, we have 
$$
d_1+ \cdots+ d_k \leq  d+k+1.
$$

\subsection*{Acknowledgments}
We are grateful to   Tatsuo Suwa,  Marcio Gomes  Soares and Renato Vidal Martins    for interesting conversations.
The authors also thank the anonymous referees for giving many suggestions that helped improving the
presentation of the paper.
This work was partially supported by CNPq, CAPES, FAPEMIG and FAPESP-2015/20841-5.
We are grateful to Imecc--Unicamp for its  hospitality.


\section{Preliminaries}

\subsection{Holomorphic Pfaff Systems} 
Given a complex manifold $X$ of dimension $n$, we denote by $\Omega_X^p$ the
sheaf of germs of holomorphic $p$-forms on $X$.

\begin{defi}
Let $X$ be an $n$-dimensional complex manifold. A  holomorphic  Pfaff system of rank $p$ ($1\leq p \leq n$) on $X$ is a non-trivial section $ \omega \in \mathrm H^0(X,\Omega_X^p\otimes   \NNN)$,  where 
$\NNN$ is a  holomorphic line bundle on $X$.    The singular set of $\omega$ is defined by   $\mathrm{Sing}(\omega)=\{z\in X; \ \omega(z)=0\}$.
\end{defi}

\noindent Given a Pfaff system $\omega$ of rank $p$ on $X$, then $\omega$ is determined by the following: 
\begin{itemize}
\item [(i)] an open covering $\{U_{\alpha}\}_{\alpha\in \Lambda}$  of $X$;
\item [(ii)] holomorphic $p$-forms $\omega_{\alpha} \in \Omega^p_{U_{\alpha}}$ satisfying
\end{itemize}
\begin{eqnarray}\nonumber
\omega_{\alpha } = (h_{\alpha \beta}) \omega_{\beta}\,\,\,\,\,\,\,\,\,\,\,\,\,\,\,\,\,\,\mbox{on}\,\,\,\,\,\,\,\,\,\,\,\,\,\,\, U_{\alpha}\cap U_{\beta}\neq \emptyset,
\end{eqnarray}
\noindent where $h_{\alpha \beta}\in \OO(U_{\alpha}\cap U_{\beta})^{\ast}$ determines the cocycle representing $\NNN$. For more details on Pfaff systems see \cite{CJV, CMM,EK}.

\begin{defi}
We say that an analytic subvariety $V \subset X$ is  {\it invariant} by a Pfaff system $\omega$ if 
$
i^*\omega \equiv 0,
$
where $i: V\hookrightarrow X$ is the inclusion map.
\end{defi}
Let $\omega$ be a Pfaff system of rank $p$  on $X$  and $V$ an analytic subvariety of  $X$ of pure  codimension $k$ . Suppose that for each $\alpha\in \Lambda$ we have 
$$
V\cap U_{\alpha} =\{z\in U_{\alpha}: f_{\alpha,1}(z)=\cdots= f_{\alpha,k}(z) = 0\},
$$
where $f_{\alpha,1},\ldots, f_{\alpha,k}\in \OO(U_{\alpha})$. If $V$ is invariant by $\omega$, then for each $i \in \{1,\ldots,k\}$  there exist holomorphic  $(p+1)$-forms $\theta^{\alpha}_{i1},\ldots, \theta^{\alpha}_{ik}\in \Omega^{p+1}_{U_{\alpha}}$, such that
\begin{eqnarray}\label{3p2}
\omega_{\alpha}\wedge df_{\alpha,i} = f_{\alpha,1}\theta^{\alpha}_{i1} + \cdots + f_{\alpha,k}\theta^{\alpha}_{ik}.
\end{eqnarray}

\subsection{Pfaff Systems on $\PP^n$}

 Let $\omega \in \mathrm H^0(\PP^n,\Omega_{\PP^n}^k(r))$ be a holomorphic  Pfaff system of rank $k$ on  $\PP^n$. 
Now take a    generic non-invariant  linearly embedded subspace $i:H\simeq \PP^k \hookrightarrow \PP^n$.  We have an induced  non-trivial section  
$i^*\omega \in \mathrm H^0(H ,\Omega_{H}^k(r)) \simeq \mathrm H^0( \PP^k , \OO_{\PP^k}(-k-1+r)),$ 
since $\Omega_{\PP^k}^k=\OO_{\PP^k}(-k-1)$ .  The tangency set between $\omega$ and $H$, denoted by $Z(i^*\omega)$, is defined as the hypersurface of zeros of $i^*\omega$ on $H$.  The {\it degree} of $\omega$, denoted by $\deg(\omega)$, is defined as the degree of $Z(i^*\omega)$ in $H$ and, therefore, is given by $$\deg(\omega)=-k-1+r.$$

In particular, in this case, $\omega \in \mathrm H^0(\PP^n,\Omega_{\PP^n}^k(d+k+1))$, where $\deg(\omega) = d$. A Pfaff system of degree $d$ can be induced  by  a polynomial $k$-form on $\CC^{n+1}$ with homogeneous coefficients of degree $d+1$, see for instance \cite{CMM,CMM2}.

\subsection{GSV-index on Surfaces}  

In this section we present  Brunella's definition of the GSV-index for one-dimensional  holomorphic  foliations on surfaces, see  \cite{Bru}. 

Let $X$ be a complex compact  surface and  $\fol$ a  one-dimensional holomorphic foliation on $X$. Let $C$ be  a reduced curve on $X$. 
Consider   $ \omega \in \mathrm H^0(X,\Omega_X^1\otimes   \NNN)$ a  rank one  Pfaff system inducing $\fol$. 
If  $C$ is invariant by $\fol$ we say that  $\fol$ is   {\it logarithmic along} $C$. 

Given a point $x\in C$, let $f=0$  be a local equation of $C$ in a neighborhood $U_{\alpha}$ of $x$ and let $\omega_{\alpha}$ be the holomorphic $1$-form inducing the foliation   $\fol$ on $U_{\alpha}$. Since $\fol$ is logarithmic along $C$,  it follows from   \cite{Sai, LN, Suw} that there are holomorphic functions $g$ and $\xi$ defined in a neighborhood of $x$,  
that do not vanish identically both of them
simultaneously on $C$, such that
\begin{eqnarray} \label{2p11}
g\,\frac{\omega_{\alpha}}{f} = \xi\, \frac{df}{f} + \eta,
\end{eqnarray}
\noindent with $\eta$ being a suitable holomorphic $1$-form.
M. Brunella in  \cite{Bru} showed that   the GSV-index can be defined as follows:

\begin{defi}[Brunella \cite{Bru}] Let $\fol$ be a one-dimensional holomorphic foliation on a complex compact surface $X$ and logarithmic along a reduced curve $C \subset X$. Given $x\in C$, we define
$$
\mathrm{GSV}(\fol, C, x) = \sum_i \mathrm{ord}_x \left(\displaystyle\frac{\xi}{g}|_{C_i}\right),
$$
 where $C_i\subset C$ are irreducible components of $C$ and  $\mathrm{ord}_x \left(\displaystyle\frac{\xi}{g}|_{C_i}\right)$ denotes the  order of vanishing of $\displaystyle\frac{\xi}{g}|_{C_i}$ at  $x$.   \end{defi}

\begin{teo} [ Brunella \cite{Bru}] Let $\fol$ be a  one-dimensional holomorphic foliation on a complex compact surface $X$ and logarithmic along a reduced curve $C \subset X$. Then
\begin{eqnarray}\nonumber
\sum_{x\in Sing\left(\fol\right) \cap C} \mathrm{GSV}(\fol,C,x) =  \NNN  \cdot  C - C \cdot C\,.
\end{eqnarray}

\end{teo}

\subsection{Decomposition of meromorphic forms}

A. G. Aleksandrov in  \cite{Ale}  introduced  the concept of {\it multiple residues} of a logarithmic differential form with poles along a complete intersection which is a   generalization of   Saito's  residues  \cite{Sai}. 

In this section, we  make a brief presentation  about a   decomposition of meromorphic forms with poles  along complete intersections which  is  used in the definition of the  GSV-index for Pfaff systems.

Let $U$ be a germ of $n$-dimensional complex manifold. Let $D$ be an analytic reduced hypersurface  on $U$ and 
consider its decomposition
into  irreducible components  
$$
D = D_1\cup\cdots \cup D_k,
$$
and suppose that   the analytic subvariety  $V = D_1\cap\cdots \cap D_k$ has pure codimension $k$. We assume  that  
$$
V  = \{z\in U : f_{ 1}(z) =\cdots =f_{ k}(z) = 0\},
$$
 with $f_{1},\ldots,f_{  k}\in \OO(U )$ and for each $i \in \{1,\ldots, k \}$,
$$
D_i   = \{z\in U : f_{i}(z) = 0\}.
$$
\noindent Since $V$ is  a  reduced variety, then the   $k$-form $df_{1}\wedge \ldots \wedge df_{ k}$ is not identically zero on each irreducible component of $V$.

We denote by $\Omega^q_U(\hat{D}_i)$, $q\geq 1$, the $\OO_U$-module of meromorphic differential $q$-forms  with simple poles on the $\hat{D}_i = D_1\cup \cdots \cup  D_{i-1}\cup D_{i+1} \cup \cdots\cup D_k$, for each $i = 1,2,\ldots,k$.

\begin{teo}[Aleksandrov \cite{Ale}] \label{teo_alec}
Let $\omega  \in \Omega^q_U (D)$ be a meromorphic $q$-form with simple poles on $D$. If for each $j=1,\ldots, k$,
\begin{eqnarray}\nonumber
df_{ j}\wedge \omega  \in \displaystyle \sum_{i=1}^k\Omega^{q+1}_{U }(\hat{D}_i) 
\end{eqnarray}
\noindent then, there exist a holomorphic function $g $,   which is not identically zero on every irreducible component of $V$, a holomorphic $(q-k)$-form $\xi \in \Omega_{U }^{q-k}$ and a meromorphic $q$-form $\eta \in \sum^k_{i=1} \Omega^{q}_{U } (\hat{D}_i)$  
such that the following decomposition holds
\begin{eqnarray}\label{expr110}
g \omega  = \frac{df_{ 1}}{f_{ 1}}\wedge\cdots \wedge \frac{df_{ k}}{f_{ k}} \wedge  \xi  + \eta .
\end{eqnarray}
\end{teo}

\hyphenation{res-tric-tion}

Since $g $ is not identically zero on every irreducible component of $V$, the restriction  $$\displaystyle \frac{\xi }{g }|_{V}$$  is well defined and it is called  the  {\it multiple residue} of the meromorphic $q$-form $\omega $.

\begin{prop} [Aleksandrov \cite{Ale}] \label{propx}
The multiple residues  of the  meromorphic $q$-form $\omega$ do not depend on the  decomposition (\ref{expr110}).
\end{prop}

\begin{obs}

It follows from  \cite[remark 2]{Ale} that in the  decomposition (\ref{expr110}), the function $g $ belongs to the ideal of $\OO(U )$ generated by all $k\times k$ minors of the  Jacobian matrix 
$\mathrm{Jac}(f_1,\ldots f_{ k})$ of the  map
$$
z\in U\longmapsto (f_{ 1}(z),\ldots, f_{  k}(z))\in \CC^k.
$$
 In other words, if for each multi-index $I = (i_1,\ldots,i_k)$, $1\leq i_1,\ldots,i_k\leq n$, we denote the corresponding minor of  $ \mathrm{Jac}(f_{ 1},\ldots f_ {k})$ by 
\begin{eqnarray}\nonumber
\Delta_{I} = \det \left[ \frac{\partial f_{ i}}{\partial z_{i_r}}\right ], \,\,\,\,\,1\leq i, r\leq k,
\end{eqnarray}
\noindent then $g $ can be written as
\begin{eqnarray}\label{3p8}
g = \sum_{\mid I\mid = k}\lambda_I\, \Delta_{I}, \,\,\,\,\,\mbox{ $\lambda_I\in \OO(U )$.}
\end{eqnarray}
\noindent Moreover, if the meromorphic $q$-form $\omega $ is represented in $U $ by
\begin{eqnarray}\nonumber
\omega  = \frac{1}{f_{1} \cdot \ldots \cdot f_{ k}}\sum_{\mid J\mid = q}a_J(z)\, dZ_J, \,\,\,\,\,\mbox{$a_J\in \OO(U )$ },
\end{eqnarray}
\noindent then for each minor $\Delta_{I}$  the following identity holds:
\begin{eqnarray}\label{3p9}
\Delta_{I}  \sum_{\mid J\mid = q}a_J \, dZ_J = \\\nonumber = 
\nonumber df_{ 1}\wedge \cdots \wedge df_{ k} \wedge \left( \sum_{\mid I'\mid = q-k}a_{(I,I')}\,dZ_{I'}\right)  +  (f_{ 1}  \cdots   f_{ k})\, \eta .
\end{eqnarray}
\noindent Thus, for the special case where $q = k$, we have that $\xi  \in \OO(U )$ is given by
\begin{eqnarray}\label{3p7}
\xi = \sum_{\mid I\mid = k}\lambda_I\, a_{I} \,.
\end{eqnarray}

\end{obs}

\begin{prop}\label{prop_dec}
Let $\omega\in \mathrm H^0(X,\Omega_X^p\otimes \NNN)$ be  a Pfaff system of rank $p$  on a complex manifold $X$, 
and  $V\subset X$ a reduced local complete intersection subvariety  of codimension $k$ which is  invariant by $\omega$.
Then for all local representations  $\omega_{\alpha}= \omega|_{U_\alpha}$  of $\omega$,  and all local expressions  of $V$ in $U_\alpha$
$$
V\cap U_{\alpha} = \{z\in U_{\alpha}: f_{\alpha,1}(z)=\cdots= f_{\alpha,k}(z) = 0 \},
$$
 there exist a holomorphic function $g_{\alpha}\in\OO(U_{\alpha})$, a holomorphic $(p-k)$-form $\xi_{\alpha}\in\Omega^{p-k}_{U_{\alpha}}$ and a holomorphic $p$-form $\eta_{\alpha}\in\Omega^{p}_{U_{\alpha}}$, such that
\begin{eqnarray}\label{expr1110}
g_{\alpha}\, \omega_{\alpha} = df_{\alpha,1}\wedge\cdots \wedge df_{\alpha,k} \wedge  \xi_{\alpha} + \eta_{\alpha}.
\end{eqnarray}
\noindent Moreover, $g_{\alpha}$  is not identically zero on every irreducible component of  $V$ and $\eta_{\alpha}$ is given by
\begin{eqnarray}\nonumber
\eta_{\alpha} = f_{\alpha,1}\,\eta_{\alpha,1}+\cdots+f_{\alpha,1}\,\eta_{\alpha,k},
\end{eqnarray}
\noindent where each $\eta_{\alpha,i}\in\Omega^{p}_{U_{\alpha}}$ is a holomorphic $p$-form.

\end{prop}
\begin{proof} Consider for each  $i \in \{1,\ldots, k\}$ 
\begin{eqnarray}\nonumber
D_i = \{z\in U_{\alpha}: f_{\alpha,i}(z) = 0\},
\end{eqnarray}
and 
\begin{eqnarray}\nonumber
\hat{D}_i = D_1 \cup \cdots \cup D_{i-1} \cup D_{i+1}\cup \cdots \cup D_k.
\end{eqnarray}
\noindent Since $V$ is invariant by $\omega$ it  follows from the  expression (\ref{3p2}) that for each $i \in \{1,\ldots, k\}$, there exist differential $(p+1)$-forms $\theta^{\alpha}_{i1},\ldots, \theta^{\alpha}_{ik}\in
 \Omega^{p+1}_{U_{\alpha}}$, such that
\begin{eqnarray}\nonumber
\omega_{\alpha}\wedge df_{\alpha,i} = f_{\alpha,1}\theta^{\alpha}_{i1} + \cdots + f_{\alpha,k}\theta^{\alpha}_{ik}.
\end{eqnarray}
\noindent  With this, we deduce that
\begin{eqnarray}\nonumber
df_{\alpha,j} \wedge \frac{\omega_{\alpha}}{f_{\alpha}} \in \sum^k_{i=1} \Omega^p_{U_{\alpha}}(\hat{D}_i),\,\,\,\,\,\,\,\,\,\,\,\,\,\,j=1,\ldots,k.
\end{eqnarray}
\noindent Hence, the meromorphic $p$-form $\displaystyle\frac{\omega_{\alpha}}{f_{\alpha}}$ satisfies the hypothesis of Theorem \ref{teo_alec} and the decomposition  ($\ref{expr1110}$) follows from decomposition (\ref {expr110}).
\end{proof}

\noindent The decomposition (\ref{expr1110}) will be called    an {\it  Aleksandrov-Saito's   decomposition of $\omega$} in $U_{\alpha}$.

\section{GSV-index for Pfaff system on projective manifolds}\label{section-coro}

\noindent In this section we define the GSV-index for Pfaff systems on projective manifolds $X$.    

Let $X$ be a projective manifold. Consider a Pfaff system   $\omega\in \mathrm H^0(X,\Omega_X^p\otimes \NNN)$, of rank $p$,  and $V$ a reduced  local complete intersection  subvariety   of pure  codimension $k$ invariant by $\omega$. Let us denote  $\mathrm{Sing}(\omega,V): =\mathrm{Sing}(\omega)\cap V  $. 
We   also assume that the rank of $\omega$ coincides with the codimension of $V$, i.e., $p=k$.
Fixed an irreducible component $S_i$ of  $\mathrm{Sing}(\omega,V)$,   let us take $\omega_{\alpha}= \omega|_{U_\alpha}$, a local representation of $\omega$,  such that $U_{\alpha}\cap S_i \neq \emptyset$. Assume that $\omega_{\alpha}$ is defined by
\begin{eqnarray}\nonumber
\omega_{\alpha} = \sum_{\mid I\mid = k}a_I(z)dZ_I, \,\,\,\,\,\mbox{$a_I\in \OO(U_{\alpha})$}.
\end{eqnarray}
\noindent Also, we  consider  a local expression of $V$ in $U_{\alpha}$, given by\\
$$
V\cap U_{\alpha} = \{z\in U_{\alpha}: f_{\alpha,1}(z)=\cdots= f_{\alpha,k}(z) = 0 \}.
$$
\noindent and let us take an Aleksandrov-Saito's   decomposition  of $\omega$  in $U_{\alpha}$,

\begin{eqnarray}\label{2p10}
g_{\alpha}\,\omega_{\alpha} = (df_{\alpha,1}\wedge\cdots \wedge df_{\alpha,k})\, \xi_{\alpha} + \eta_{\alpha},
\end{eqnarray}
\noindent with $\eta_{\alpha} = f_{\alpha,1}\,\eta_{\alpha,1} + \cdots +f_{\alpha,k}\,\eta_{\alpha,k}$, where $\eta_{\alpha,i}\in \Omega^{k}_{U_{\alpha}}$ and, furthermore, $\xi_{\alpha}$ being a holomorphic function.

In this context, we define the GSV-index:

\begin{defi}Suppose that $S$ is a codimension one subvariety of $V$.  The GSV-index of $\omega$ relative to $V$ in $S$ is defined  by 
\begin{eqnarray}\nonumber
\mathrm{GSV}(\omega, V, S) : = \sum_j \mathrm{ord}_{S}\left(\displaystyle\frac{\xi_{\alpha}}{g_{\alpha}}|_{V_j}\right),
\end{eqnarray}
where the sum is taken over all irreducible components $V_j $ of $V$ and 
 $\mathrm{ord}_{S}\left(\displaystyle\frac{\xi_{\alpha}}{g_{\alpha}}|_{V_j}\right)$ denotes  the order of vanishing of $\frac{\xi_{\alpha}}{g_{\alpha}}|_{V_j}$ along   $S$.
 \end{defi}
We recall  that for a rational function $r$  the order of vanishing is defined by  $\mathrm{ord}_{S}(r) =l_{\mathcal{O}_{S,V_j}}(\mathcal{O}_{S,V_j}/(r))$, 
where $l_{\mathcal{O}_{S,V_j}}(\mathcal{O}_{S,V_j}/(r))$ denotes the length of the   $\mathcal{O}_{S,V_j}$-module $ \mathcal{O}_{S,V_j}/(r)$, see 
 \cite{Fu}. 

Now we will prove our main result. 
\subsection{Proof of Theorem \ref{prop9}}

Firstly,    it  follows  from the definition and  Proposition  \ref{propx} that $\mathrm{GSV}(\omega, V, S_i)$ does not depend on chosen decomposition of $\omega$.
Now, we will prove that $\mathrm{GSV}(\omega, V, S_i)$ does not depend on the local representation $\omega_{\alpha}$ of $\omega$, and does not depend on local expression for  $V$:
$$
V\cap U_{\alpha} = \{z\in U_{\alpha}: f_{\alpha,1}(z)=\cdots= f_{\alpha,k}(z) = 0 \}.
$$
In fact, if we consider another local representation $\omega_{\beta} = \omega|_{U_{\beta}}$, such that $U_{\beta}\cap S_i \neq \emptyset$ and other local expression  for $V$
$$
V\cap U_{\beta} = \{z\in U_{\beta}: f_{\beta,1}(z)=\cdots= f_{\beta,k}(z) = 0 \},
$$
we obtain the Aleksandrov-Saito's decomposition of $\omega$  in $U_{\beta}$
\begin{eqnarray}\label{expr6}
g_{\beta}\,\omega_{\beta} = (df_{\beta,1}\wedge\cdots \wedge df_{\beta,k})\, \xi_{\beta} + \eta_{\beta}
\end{eqnarray}
and the decomposition of $\omega$  in $U_{\alpha}$
\begin{eqnarray}\label{expret01}
g_{\alpha}\,\omega_{\alpha} = (df_{\alpha,1}\wedge\cdots \wedge df_{\alpha,k})\, \xi_{\alpha} + \eta_{\alpha},
\end{eqnarray}
where $ \eta_{\beta}|_{V}=\eta_{\alpha}|_{V}=0$. 
\noindent Now, in the intersection $U_{\alpha}\cap U_{\beta}\neq \emptyset$ we have that 
\begin{eqnarray}\label{11}
(df_{\alpha,1}\wedge\ldots \wedge df_{\alpha,k})  = m_{\alpha\beta}\left(df_{\beta,1}\wedge\ldots \wedge df_{\beta,k}\right)+ \theta_{\alpha,\beta},
\end{eqnarray}
 where $m_{\alpha\beta}|_{V}\in\OO_V(U_{\alpha}\cap U_{\beta} \cap V)^*$  is the  cocycle of the    determinant of the normal bundle $\det (N_{V/X})$ on $V$ and $\theta_{\alpha,\beta}$ is a holomorphic $k$-form such that 
$ \theta_{\alpha,\beta}|_{V}=0$. Also, in  $U_{\alpha}\cap U_{\beta}$ we have that 
\begin{eqnarray}\label{12}
\omega_{\alpha} = h_{\alpha \beta} \omega_{\beta},
\end{eqnarray}
  where $h_{\alpha \beta}, \in\OO(U_{\alpha}\cap U_{\beta})^*$  is the  cocycle of the line bundle $\NNN$. 
 On the one hand, by using (\ref{11})  and (\ref{12}) in (\ref{expret01}), we obtain 
\begin{eqnarray}\label{expr5}
(g_{\alpha}h_{\alpha \beta})\omega_{\beta} = (m_{\alpha\beta} \xi_{\alpha})(df_{\beta,1}\wedge\ldots \wedge df_{\beta,k})+  \eta_{\alpha} +  \theta_{\alpha,\beta}.
\end{eqnarray}
\noindent On the other hand, by  using  ($\ref{expr6}$) in ($\ref{expr5}$), we obtain 
\begin{eqnarray}\nonumber
( h_{\alpha \beta} \xi_{\beta}g_{\alpha}  -     m_{\alpha\beta} \xi_{\alpha}  g_{\beta})    \left( df_{\beta,1}\wedge\ldots \wedge df_{\beta,k}\right) = \eta_{\beta}+g_{\beta}\eta_{\alpha} +  g_{\beta}\theta_{\alpha,\beta}. 
\end{eqnarray}
Since $ ( \eta_{\beta}+g_{\beta}\eta_{\alpha} +  g_{\beta}\theta_{\alpha,\beta})|_{V}\equiv 0$, we conclude that 
\begin{eqnarray}\nonumber
( h_{\alpha \beta} \xi_{\beta}g_{\alpha}  -     m_{\alpha\beta} \xi_{\alpha}  g_{\beta})    \left( df_{\beta,1}\wedge\ldots \wedge df_{\beta,k}\right) \equiv 0 (\mod(f_{\beta,1},\ldots, f_{\beta,k})).
\end{eqnarray}
It follows from    \cite{S2}  that there exists $r\in \mathbb{Z}$, with $r\geq 1$, such that
$$
\D^r ( h_{\alpha \beta} \xi_{\beta}g_{\alpha}  -     m_{\alpha\beta} \xi_{\alpha}  g_{\beta})    \in  df_{\beta,1}\wedge \Omega_{U_{\beta}}^{k-1}+\cdots +df_{\beta,k}\wedge \Omega_{U_{\beta}}^{k-1},
$$
where $\D$ is the ideal of $\OO_{U_{\beta}}$    generated by all minors of maximal order of the Jacobian matrix Jac$(f_{\beta,1}, \dots,f_{\beta,k} ). $
Now, as observed   in \cite[Proposition 2]{Ale}  the image $\mathrm{Im}(\D)$ of $\D$ in the ring $\OO_{V\cap U_{\beta}}$ is not   equal to Ann$(\OO_{V\cap U_{\beta}})$. 
In fact,  since $V$ is reduced, then $df_{\beta,1}\wedge\ldots \wedge df_{\beta,k}$
does not vanish identically on each irreducible component of $V$.
Therefore,  it follows from \cite[Theorem 2.4. (1)]{BR} that the 
 $\OO_{V\cap U_{\beta}}$-depth of the ideal $\D^r $ is at least $1$. Then, there is $D\in \D$ which
is not a zero-divisor in  $\OO_{V\cap U_{\beta}}$ and
$$
D^r ( h_{\alpha \beta} \xi_{\beta}g_{\alpha}  -     m_{\alpha\beta} \xi_{\alpha}  g_{\beta})   \in  df_{\beta,1}\wedge \Omega_{U_{\beta}}^{k-1}+\cdots +df_{\beta,k}\wedge \Omega_{U_{\beta}}^{k-1}.
$$
Thus, the class $D^r ( h_{\alpha \beta} \xi_{\beta}g_{\alpha}  -     m_{\alpha\beta}  \xi_{\alpha}  g_{\beta}) =0 \in \Omega_{V\cap U_{\beta}}^k$. Then  
 $$( h_{\alpha \beta} \xi_{\beta}g_{\alpha}  -     m_{\alpha\beta} \xi_{\alpha}  g_{\beta})|_{V} =0$$
 since $D\in \D$  
is not a zero-divisor.  
We conclude  that 
\begin{eqnarray}\label{expr7}
\displaystyle \frac{\xi_{\alpha}}{g_{\alpha}}|_{V} = h_{\alpha \beta}\,(m_{\alpha\beta})^{-1}\,\displaystyle\frac{\xi_{\beta}}{g_{\beta}}|_{V}.
\hyphenation{su-pport}
\end{eqnarray}
Thus, we obtain a    family of meromorphic    functions ${\frak s}= \{ (\xi_{\alpha}/g_{\alpha})|_V\}_{\alpha \in \Lambda}$ which  defines a rational  section of the line bundle $[\NNN \otimes \det (N_{V/X})^{-1}]|_V$ on $V$ whose support is $\mathrm{Sing}(\omega,V)=\cup_i S_i $ . 
Since  $h_{\alpha \beta}$ and $(m_{\alpha\beta})^{-1}$ are non-vanishing holomorphic functions we  observe that  
\begin{eqnarray}\nonumber
\mathrm{ord}_{S_i}\left(\displaystyle\frac{\xi_{\alpha}}{g_{\alpha}}|_{V}\right) = \mathrm{ord}_{S_i}\left(\displaystyle\frac{\xi_{\beta}}{g_{\beta}}|_{V}\right).
\end{eqnarray}
Therefore, we have  the Cartier  divisor  associated to the section ${\frak s}$ given by 
\begin{eqnarray}\nonumber
({\frak s})_0 = \sum_i \mathrm{GSV}(\omega,V,S_i)[S_i].
\end{eqnarray}
That is
\begin{eqnarray}\nonumber
\OO(({\frak s})_0) \cong [\NNN \otimes \det (N_{V/X})^{-1}]|_V.
\end{eqnarray}
\noindent Thus, by normalization property of the Chern class \cite[Thm. 3.2,(f)]{Fu}  we have 
\begin{eqnarray}\nonumber
 c_1([ \NNN \otimes \det (N_{V/X})^{-1}]|_V) &=& c_1(\OO(({\frak s})_0)) \\\nonumber\\\nonumber   &=&  \sum_i \mathrm{GSV}(\omega,V,S_i)[S_i]. 
\end{eqnarray}
Therefore
\begin{eqnarray}\nonumber
\sum_i \mathrm{GSV}(\omega,V,S_i)[S_i] =  c_1([ \NNN \otimes \det (N_{V/X})^{-1}]|_V)  \frown [V].
\end{eqnarray}

The next result says us how   to calculate the GSV-index which  can be compared with  Suwa's  formula in \cite[Proposition 5.1 ]{Suw3}.

\begin{teo}\label{teo404}
For each multi-index $I$, with $\mid I \mid = k$, the following formula holds 
\begin{eqnarray}\nonumber
\mathrm{GSV}(\omega,V,S_i) = \mathrm{ord}_{S_i}(a_{I\,}|_V) - \mathrm{ord}_{S_i}(\Delta_{I\,}|_V).
\end{eqnarray}
\noindent where $\Delta_I$ is the $k\times k$  minor of the Jacobian matrix   $ \mathrm{Jac}(f_{\alpha,1},\ldots,f_{\alpha,k})$, corresponding to the multi-index $I$.
\end{teo}
\begin{proof} By the  expression (\ref{3p9}) we have 
\begin{eqnarray}\nonumber
(\Delta_{I} . \sum_{\mid J\mid = k}a_J \, dZ_J)|_V = [ (df_{\alpha,1}\wedge \ldots \wedge df_{\alpha,k})\, a_{I}]|_V.
\end{eqnarray}
\noindent Since
\begin{eqnarray}\nonumber
df_{\alpha,1}\wedge \ldots \wedge df_{\alpha,k} = \sum_{\mid J\mid = k}\Delta_J \, dZ_J,
\end{eqnarray}
\noindent  for each $J$, with $\mid J\mid = k$,  we get 
\begin{eqnarray}\label{3p10}
(\Delta_I \, a_J)|_V = (\Delta_J \, a_I)|_V.
\end{eqnarray}
\noindent Thus,
\begin{eqnarray}\nonumber
\mathrm{GSV}(\omega, V, S_i) &=& \mathrm{ord}_{S_i}\left(\frac{\xi_{\alpha}}{g_{\alpha}}|_V\right)\\\nonumber\\\nonumber\\\nonumber
& = & \mathrm{ord}_{S_i}\left(\frac{\xi_{\alpha}}{g_{\alpha}}|_V\right) + \mathrm{ord}_{S_i}(\Delta_I|_V) - \mathrm{ord}_{S_i}(\Delta_I|_V)  \\\nonumber\\\nonumber\\\nonumber
& = & \mathrm{ord}_{S_i}\left(\frac{\xi_{\alpha}\,\Delta_I}{g_{\alpha}}|_V\right) - \mathrm{ord}_{S_i}(\Delta_I|_V).
\end{eqnarray}
\noindent By (\ref{3p7}) we have 
\begin{eqnarray}\nonumber
\xi_{\alpha} = \sum_{\mid J\mid = k}\lambda_J \, a_J.
\end{eqnarray}
\noindent Hence, by using (\ref{3p10}) we obtain
\begin{eqnarray}\nonumber
\xi_{\alpha}\,\Delta_I|_V &=& \sum_{\mid J\mid = k}\lambda_J \, (a_J\,\Delta_I)|_V=\sum_{\mid J\mid = k}\lambda_J \, (\Delta_J\,a_I)|_V\\\nonumber\\\nonumber\\\nonumber
&=& (\sum_{\mid J\mid = k}\lambda_J \, \Delta_J)\,a_I |_V=  g_{\alpha}\,a_I |_V,    
\end{eqnarray}
\noindent where in the last step we have used (\ref{3p8}).
Thus,
\begin{eqnarray}\nonumber
\mathrm{ord}_{S_i}\left(\frac{\xi_{\alpha}\,\Delta_I}{g_{\alpha}}|_V\right) = \mathrm{ord}_{S_i}\left(a_{I\,}|_V\right).
\end{eqnarray}
Then 
\begin{eqnarray}\nonumber
\mathrm{GSV}(\omega,V,S_i) = \mathrm{ord}_{S_i}(a_{I\,}|_V) - \mathrm{ord}_{S_i}(\Delta_{I\,}|_V).
\end{eqnarray}
\end{proof}

\begin{cor}\label{4p4}
If $S_i\cap \mathrm{Sing}(V)= \emptyset$, then $\mathrm{GSV}(\omega,V,S_i)\geq 0$.
\end{cor}
\begin{proof} By Theorem \ref{teo404}, 
for each multi-index
 $I$, with $\mid I \mid = k$, we have 
\begin{eqnarray}\label{4p3}
\mathrm{GSV}(\omega,V,S_i) = \mathrm{ord}_{S_i}(a_{I\,}|_V) - \mathrm{ord}_{S_i}(\Delta_{I\,}|_V).
\end{eqnarray}
\noindent Let $h_i\in \OO_V(U_{\alpha} \cap V)$ be a function which defines locally $S_i$ in $U_{\alpha}\cap V$, i.e.,
\begin{eqnarray}\nonumber
(U_{\alpha}\cap V)\cap S_i = \{z\in U_{\alpha}\cap V: h_i(z) = 0\}.
\end{eqnarray}
\noindent If $\mathrm{GSV}(\omega,V,S_i) < 0$ then, by (\ref{4p3}), for each multi-index $I$, with $\mid I \mid = k$,
\begin{eqnarray}\nonumber
\mathrm{ord}_{S_i}(\Delta_{I\,}|_V) = \delta_I > 0.
\end{eqnarray}
\noindent Therefore, we obtain 
\begin{eqnarray}\nonumber
\Delta_{I}|_V = h_i^{\delta_I}\,\mu_I|_V,
\end{eqnarray}
\noindent for some function $\mu_I\in \OO_V(U_{\alpha} \cap V)$.
This implies that for each multi-index $I$, with $\mid I \mid = k$, the following inclusion occurs
\begin{eqnarray}\nonumber
\{z\in U_{\alpha}\cap V: h_i(z) = 0\} \subset \{z\in U_{\alpha}\cap V: \Delta_{I}(z) = 0\}.
\end{eqnarray}
\noindent Thus, we conclude  that
\begin{eqnarray}\nonumber
(U_{\alpha}\cap V)\cap S_i = \{z\in U_{\alpha}\cap V: h_i(z) = 0\} \subset \bigcap_{\mid I \mid = k}\{z\in U_{\alpha}\cap V: \Delta_{I}(z) = 0\} = U_{\alpha} \cap \mathrm{Sing}(V)
\end{eqnarray}
\noindent and consequently, we get $S_i\cap \mathrm{Sing}(V)\neq \emptyset$. 
\end{proof}

\begin{cor}\label{4p10}
If $V$ is  smooth, then $\mathrm{GSV}(\omega,V,S_i)>0$.
\end{cor}

\begin{proof} If $V$ is smooth, then there exists some $k\times k$ minor $\Delta_I$ (of Jacobian matrix Jac$(f_{\alpha,1},\ldots,f_{\alpha,k})$) which  is not   zero along $V$ and in this way, we obtain
\begin{eqnarray}\label{4p7}
\mathrm{ord}_{S_i}(\Delta_I|_V) = 0.
\end{eqnarray}
Therefore,  it follows from  Theorem \ref{teo404} that 
$$
\mathrm{GSV}(\omega,V,S_i) = \mathrm{ord}_{S_i}(a_{I\,}|_V) - \mathrm{ord}_{S_i}(\Delta_{I\,}|_V)=\mathrm{ord}_{S_i}(a_{I\,}|_V)>0.
$$
\end{proof}

\section{An Application: Poincar\'e problem for Pfaff systems } \label{sec23}
In this section we show that the non-negativity of   the GSV-index gives an  obstruction to the solution of the Poincar\'e problem for Pfaff systems on projective spaces. This application was motivated  by a result due to  M. Brunella  in  \cite{Bru}.

\begin{teo}\label{teo_ult}
Let $\omega \in \mathrm H^0(\PP^n,\Omega_{\PP^n}^k(d+k+1))$ be a holomorphic  Pfaff system of rank $k$ and degree $ d$. Let $V\subset \PP^n$  be a reduced   complete intersection variety, of codimension $k$ and multidegree $(d_1,\dots,d_{k})$ ,  invariant  by $\omega$. Suppose that  $\mathrm{Sing}(\omega,V)$ has codimension one in $V$,  then
\begin{eqnarray} \nonumber 
\sum_i \mathrm{GSV}(\omega,V,S_i)\deg(S_i) \,\,\,=\,\,\, [d+k+1 - (d_1+\cdots+d_k)]\,\cdot (d_1\cdots d_k),
\end{eqnarray}
where $S_i$ denotes an irreducible component of $\mathrm{Sing}(\omega,V)$.
In particular,  if $\mathrm{GSV}(\omega,V,S_i)\geq 0$, for all $i$, we have 
$$
d_1+ \cdots+ d_k \leq  d+k+1.
$$
\end{teo}
\bigskip
\begin{proof} \noindent By applying    Theorem \ref{prop9}   and taking degrees we obtain
\begin{eqnarray}\nonumber
\sum_i \mathrm{GSV}(\omega,V,S_i)\deg(S_i) &=& \deg\left(c_1([ \NNN \otimes \det ( N_{V/\PP^n})^{-1}]|_V \frown [V]\right).
\end{eqnarray}
\noindent  On the one hand,   the normal bundle of $V$ is given by
$$
N_{V/\PP^n}= \OO_{\PP^n}(d_1)\oplus\cdots \oplus \OO_{\PP^n}(d_k)|_{V}.
$$
Thus 
\begin{eqnarray}\nonumber
\det(N_{V/\PP^n}) = \OO_{\PP^n}(d_1+\dots+d_k)|_V.
\end{eqnarray}
On the other hand, we have that 
$$
\NNN = \OO_{\PP^n}(d+k+1).
$$
Now, since 
$$
[V] = (d_1 \cdots d_k)\,\,c_1(\OO(1))^k
$$
we get
\begin{eqnarray}\nonumber
\deg([\, \NNN \otimes  \det(N_{V/\PP^n})^{-1}\,]|_V  \frown [V])= [(d+k+1) - (d_1+\cdots+d_k)] \,\cdot (d_1\cdots d_k).
\end{eqnarray}
\noindent We obtain,
\begin{eqnarray}\nonumber
\sum_i \mathrm{GSV}(\omega,V,S_i)\deg(S_i) &=& [d+k+1 - (d_1+\cdots+d_k)]\,\cdot (d_1\cdots d_k).
\end{eqnarray}
Now, if $\mathrm{GSV}(\omega,V,S_i)\geq 0$, for all $i$, we have 
\begin{eqnarray}\nonumber
0 \leq \sum_i \mathrm{GSV}(\omega,V,S_i)\deg(S_i) &=& [d+k+1 - (d_1+\cdots+d_k)]\,\cdot (d_1 \cdots   d_k).
\end{eqnarray}
This implies that 
$$
d_1+ \cdots+ d_k \leq   d+k+1.
$$
\end{proof}

\hyphenation{obs-truc-tion} 

In the next result we obtain a bound similar to those by   Esteves--Cruz   \cite{EstCruz} and  Corr\^ea--Jardim \cite{CJ}.  

\begin{cor}
Let $\omega \in \mathrm H^0(\PP^n,\Omega_{\PP^n}^k(d+k+1))$ be a holomorphic  Pfaff system of rank $k$ and degree $ d$. Let $V\subset \PP^n$  be a  reduced complete intersection variety of codimension $k$ and multidegree $(d_1,\dots,d_{k})$   invariant  by $\omega$. Suppose that  $\mathrm{Sing}(\omega,V)$ has codimension one in $V$ and that 
  $S_i\cap  \mathrm{Sing}(V)= \emptyset$ for all $i$, then
\begin{eqnarray}\nonumber
d_1+ \cdots+ d_k \leq  d+k+1.
\end{eqnarray}
Moreover,   if $V$ is regular we have 
\begin{eqnarray}\nonumber
d_1+ \cdots+ d_k \leq  d+k.
\end{eqnarray}
\end{cor}

\begin{proof} 
The result follows from  Corollary \ref{4p10} and Corollary \ref{4p4}.
\end{proof}

Now we  give an optimal example.  
\begin{exe}
Consider the  Pfaff system 
  $\omega \in \mathrm H^0(\PP^n, \Omega_{\PP^n}^k(d+k+1))$
given by 
$$
\omega = \sum_{0 \le j
\le k} (-1)^j d_j f_j \ df_0 \wedge \dots \wedge \widehat
{df_j} \wedge \dots \wedge df_k,
$$
where $f_j  $ is a homogeneous polynomial
of degree $d_j$. We can see    that 
$$
d_0+d_1 +\dots + d_k = d + k + 1.
$$
Suppose that  $\deg(f_0)=d_0=1$ and that $V=\{ f_1=\cdots=f_{k }=0\}$  is smooth. We have that $V$ is  invariant by $\omega$ and  
$$
d_1 +\dots + d_k = d + k .
$$
\end{exe}


\begin{thebibliography}{99}
\bibitem{Ale} A.G. Aleksandrov, {\it Multidimensional residue theory and the logarithmic De Rham Complex}, Journal of Singularities,   \textbf{5} (2012), 1-18. 
%


\bibitem{a03} J-P. Brasselet, J.  Seade, T.  Suwa, {\it An explicit cycle representing the Fulton-Johnson class}, Singularit\'es Franco-Japonaises, Smin. Congr., Soc. Math. France, Paris,  \textbf{10} (2005) 21-38. 

\bibitem{BarSaeSuw} J-P. Brasselet, J.  Seade,  T. Suwa, {\it Vector Fields on Singular Varieties}, Lecture Notes in Mathematics, Spring, (2009).



\bibitem{Bru} M. Brunella, {\it Some remarks on indices of holomorphic fields}, Publicacions Matem\`atiques, \textbf{41} (1997), 527-544. 
%
\bibitem{BruMe} M. Brunella, L. G. Mendes, {\it Bounding the degree of Solutions to Pfaff Equations}, Publicacions Matem\`atiques, \textbf{44}, (2000), 593-604. 
%


\bibitem{BR}
D. A. Buchsbaum, D. S. Rim, {\it A generalized Koszul complex. II. Depth and multiplicity},
Trans. Amer. Math. Soc. \textbf{111} (1964), 197-224. 

\bibitem{CS} C. Camacho,  P.  Sad,  {\it  Invariant varieties through singularities of holomorphic vector fields}, Ann. of Math. \textbf{115} (1982) , 579-595. 


\bibitem{CL}
V. Cavalier,  D. Lehmann,  {\it  On the Poincar\'e inequality for one-dimensional foliations}, Compositio
Math.  \textbf{142} (2006), 529-540. 

\bibitem{Cerv1} D. Cerveau,  A. Lins Neto,  {\it Holomorphic Foliations in $\CC\PP^2$ having an invariant algebraic curve}, Annales de l'lnstitut Fourier, \textbf{41} (1991), 883-904.

\bibitem{CJ}
M. Corr\^ea Jr,  M. Jardim,  {\it Bounds for sectional genera of varieties invariant under Pfaff fields},  Illinois Journal of Mathematics, 
vol. \textbf{56} (2) (2012),   343-352. 

\bibitem{CJV} M. Corr\^ea,    M. Jardim,  R. Vidal,  \textit{  On the Singular scheme of split foliations},   Indiana University Mathematics Journal,  \textbf{64}  (5) (2015),   1359-1381. 


\bibitem{CMM}
M. Corr\^ea Jr,  L. G. Maza, M. G. Soares,  {\it Hypersurfaces Invariant by Pfaff Equations},  Communications in Contemporary Mathematics, \textbf{17} (2015), 1450051.

\bibitem{CMM2}
M. Corr\^ea Jr,  L. G. Maza, M. G. Soares, {\it Algebraic integrability of polynomial differential r-forms}. Journal of Pure and Applied Algebra, \textbf{215} (2011), 2290-2294. 

%

\bibitem{CS1} M. Corr\^ea,  
M. G. Soares,  {\it  A Poincar\'e type inequality for one-dimensional multiprojective foliations}, Bulletin Brazilian Mathematical Society, \textbf{42} (2011) 485-503. 

\bibitem{CS2} M. Corr\^ea,   
M. G. Soares,  {\it
 A note on Poincar\'e's problem for quasi-homogeneous foliations},  Proceedings of the American Mathematical Society, \textbf{140} (2012) , 3145-3150. 



%
\bibitem{EstCruz}J.  D. A. Cruz,   E. Esteves, {\it Regularity of subschemes invariant under Pfaff fields on projective spaces}. Commentarii Mathematici Helvetici \textbf{ 86} (2011),   947-965. 
%

\bibitem{EK} E. Esteves, S. L. Kleiman, \emph{Bounding Solutions of Pfaff Equations},  Comm. Algebra \textbf{31} (2003), 3771-3793. 

\bibitem{Fu}
W. Fulton, {\it Intersection Theory}, Springer, (1998).


\bibitem{a08} X. G\'omez-Mont, {\it An algebraic formula for the index of a vector field on a hypersurface
with an isolated singularity}, J. Algebraic Geom. \textbf{7} (1998) , 731-752. 
%
\bibitem{GSV}  X. G\'omez-Mont, J. Seade,   A. Verjovsky, {\it The index of a holomorphic flow with an isolated singularity}, Math. Ann. \textbf{291} (1991), 737-751. 
%

\bibitem{GH} 
P. Griffths, J. Harris, 
{\it Principles of Algebraic Geometry}. John Wiley and Sons, (1994).

\bibitem{a07} D. Lehmann, M. Soares, T. Suwa, {\it On the index of a holomorphic vector field tangent to a singular variety}, Bol. Soc. Bras. Mat. \textbf{26} (1995) 183-199. 
 
 \bibitem{LN} 
 A. Lins Neto, {\it  Algebraic solutions of polynomial differential
equations and foliations in dimension two}, in Holomorphic Dynamics,
(Mexico, 1986), Springer, Lecture Notes \textbf{1345} (1998), 
192-232. 
 
%
%
\bibitem{Poin} H. Poincar\'e, {\it Sur l'integration algebrique des equations diff\'erentielles du premier ordre et du premier degr\'e}, Rendiconti del Circolo Matematico di Palermo \textbf{ 5} (1891), 161-191.
%
%
\bibitem{Sai} K. Saito,  {\it Theory of logarithmic differential forms and logarithmic vector fields}, J. Fac. Sci. Univ. Tokyo \textbf{27}(2) (1980),  265-291.
%

\bibitem{S2}
K. Saito, {\it  On a generalization of de-Rham lemma}, Ann. Inst. Fourier (Grenoble) \textbf{26}(2) (1976),  165-170. 

\bibitem{KS}
B. Khanedani,  T. Suwa, {\it First variation of holomorphic forms and some applications}, Hokkaido Math. J. \textbf{26} (1997) , 323-335 .

\bibitem{SeaSuw1} J. Seade, T. Suwa, {\it A residue formula for the index of a holomorphic flow}, Math. Ann. \textbf{304} (1996) 621-634 .

\bibitem{SeaSuw2} J. Seade,  T. Suwa, {\it  An adjunction formula for local complete intersections}, Internat.
J. Math. \textbf{9} (1998) , 759-768.


\bibitem{Suw}
T. Suwa, {\it Indices of holomorphic vector fields relative to invariant
curves on surfaces} , Proc. Amer. Math. Soc. \textbf{123} (1995) , 2989-2997 .


\bibitem{Suw3} T. Suwa,  GSV-{\it Indices as Residues}, Jornal of Singularities \textbf{9}  (2014), 206-218 .





\bibitem{So1} M. G. Soares, {\it The Poincar\'e problem for hypersurfaces invariant by one-dimensional foliations},
Inventiones Mathematicae   \textbf{ 128}  ( 1997), 495-500. 


%



\end{thebibliography}
\end{document}